\documentclass{amsart}
\usepackage{latexsym,amsxtra,amscd,ifthen,amsmath,color}
\usepackage{amsfonts}
\usepackage{verbatim}
\usepackage{amsmath}
\usepackage{amsthm}
\usepackage{amssymb}
\usepackage[notcite,notref, final]{showkeys}
 \usepackage[all,cmtip]{xy}
 \usepackage[normalem]{ulem}

\numberwithin{equation}{section}

\theoremstyle{plain}
\newtheorem{theorem}{Theorem}[section]
\newtheorem{lemma}[theorem]{Lemma}
\newtheorem{sublemma}[theorem]{Sublemma}
\newtheorem{proposition}[theorem]{Proposition}

\newtheorem{corollary}[theorem]{Corollary}
\newtheorem{conjecture}[theorem]{Conjecture}

\theoremstyle{definition}
\newtheorem{definition}[theorem]{Definition}
\newtheorem{example}[theorem]{Example}

\newtheorem{remark}[theorem]{Remark}
\newtheorem{question}[theorem]{Question}

\newcommand{\mc}{\mathcal}
\newcommand{\Z}{\mathbb{Z}}

\newcommand{\ra}{\rightarrow}
\newcommand{\lra}{\longrightarrow}

\newcommand{\Hom}{\text{Hom}}
\newcommand{\End}{\text{End}}

\newcommand{\Specm}{\text{Specm}}
\newcommand{\Rep}{{\sf Rep}}
\renewcommand{\Vec}{{\sf Vec}}

\makeatletter              % This sequence of commands will
\let\c@equation\c@theorem  % incorporate equation numbering
\makeatother

\begin{document}

\title[Semisimple Hopf actions on commutative domains]
{Semisimple Hopf actions on commutative domains}

\author{Pavel Etingof and Chelsea Walton}

\address{Etingof: Department of Mathematics, Massachusetts Institute of Technology, Cambridge, Massachusetts 02139,
USA}

\email{etingof@math.mit.edu}

\address{Walton: Department of Mathematics, Massachusetts Institute of Technology, Cambridge, Massachusetts 02139,
USA}

\email{notlaw@math.mit.edu}

\bibliographystyle{alpha}       % Set the bibliography style to AMS

\begin{abstract}
Let $H$ be a semisimple (so, finite dimensional) Hopf algebra over an algebraically closed field $k$ of characteristic zero and let $A$ be a commutative domain over $k$. We show that if $A$ arises as an $H$-module algebra via an inner faithful $H$-action, then $H$ must be a group algebra. This answers a question of E. Kirkman and J. Kuzmanovich and partially answers a question of M. Cohen.

The main results of this article extend to working over $k$ of positive characteristic. On the other hand, we obtain results on Hopf actions on Weyl algebras as a consequence of the main theorem.
\end{abstract}


\maketitle

%\tableofcontents
%\setcounter{section}{-1}

%%%%%%%%%%%%%%%%%%%%%%%%%%%%%%%%%%%%%%%
%%%%%%%%%%%%%%%%%%%%%%%%%%%%%%%%%%%%%%%
%%%%%%%%%%%%%%%%%%%%%%%%%%%%%%%%%%%%%%%

\section{Introduction} \label{sec:intro}

We work over an algebraically closed field $k$ of characteristic zero, unless stated otherwise. Take $H$ to be a finite dimensional Hopf algebra over $k$. Let us begin by considering the following question of Miriam Cohen. 

\begin{question} \label{ques:Cohen}  \cite[Question 2]{Cohen:survey}
Let $H$ be a finite dimensional noncocommutative Hopf algebra acting on a commutative algebra $A$. Can such an action be faithful? In other words, can $A$ be a faithful left $A\#H$-module?
\end{question}

This question was answered negatively in the case that $A$ is a field  and $S^2 \neq id$, where $S$ is the antipode of $H$ \cite[Theorem~3.2]{Cohen:survey}. Its full answer remains unknown. However, the notion of a {\it faithful} Hopf algebra action is a strong condition. In this work, we focus our attention on a weaker, yet interesting notion: {\it inner faithful} Hopf actions. In other words, we want to consider Hopf ($H$-) actions that do not factor through `smaller' Hopf algebras, say $H/I$ for some nonzero Hopf ideal $I$ of $H$ (Definition~\ref{def:infaith}). One may argue that inner faithfulness is a more useful notion than faithfulness as one can pass uniquely to an inner faithful Hopf action if necessary.

It is known that there do indeed exist inner faithful actions of a {\it nonsemisimple} noncocommutative Hopf algebras on commutative algebras (see e.g. \cite{Allman}). In particular, take $H$ to be the $4$-dimensional Sweedler algebra and $A$ to be the commutative polynomial ring in two variables \cite[Section~3.2]{Allman}.
In light of this result, Ellen Kirkman and Jim Kuzmanovich proposed the following question.

\begin{question} \label{ques:KK} \cite[Question 0.8]{CWZ:Nakayama}
Suppose that $H$ is semisimple and acts inner faithfully on a commutative domain over an algebraically closed field of characteristic zero. Must $H$ be a group algebra?
\end{question}

This question was answered positively in the case that $H$ is semisimple and $A$ is a commutative polynomial ring in two variables \cite[Proposition~0.7]{CWZ:Nakayama}. Now, the main result of this article is a full affirmative answer to Question~\ref{ques:KK}, which also yields a partial answer to Question~\ref{ques:Cohen} above.

\begin{theorem} \label{thm:main} (Theorem \ref{thm:mainK})
If a cosemisimple Hopf algebra $H$ over $k$ acts inner faithfully on a  commutative domain over $k$, then $H$ is a finite group algebra.
\end{theorem}

\noindent Note that semisimplicity and cosemisimplicity are equivalent over a field of characteristic zero.

As a consequence, we also answer a question of \cite{CWWZ:filtered} pertaining to finite dimensional Hopf actions on Weyl algebras.

\begin{corollary} \label{cor:Weylintro} (Corollary~\ref{cor:Weyl})
Let $H$ be a finite dimensional Hopf algebra acting inner faithfully on the $n$-th Weyl algebra $A_n(k)$ with the standard filtration. If the $H$-action preserves the filtration of $A_n(k)$, then $H$ is a finite group algebra.
\end{corollary}

The paper is organized as follows. In Section~\ref{sec:backg}, we provide background material on Hopf algebra actions and tensor categories. Section~\ref{sec:coideal} provides results on coideal subalgebras of finite dimensional Hopf algebras. Such results are needed for the proof of the main theorem, Theorem~\ref{thm:main}, which is presented in Section~\ref{sec:proofmain}. Finally, Section~\ref{sec:directions} discusses consequences  of Theorem~\ref{thm:main}; these include Corollary~\ref{cor:Weylintro}  and versions of the main results for $k$ of positive characteristic. For instance, we have the following result.

\begin{theorem} \label{thm:maincharp} (Theorem~\ref{thm:mainKcharp}) 
Theorem~\ref{thm:main} holds if the field $k$ has characteristic $p>0$ and $K$ is a semisimple and cosemisimple Hopf algebra.
\end{theorem}

\noindent Several questions and conjectures are also posed in Section~\ref{sec:directions}.

%%%%%%%%%%%%%%%%%%%%%%%%%%%%%%%%%%%%%%%
%%%%%%%%%%%%%%%%%%%%%%%%%%%%%%%%%%%%%%%
%%%%%%%%%%%%%%%%%%%%%%%%%%%%%%%%%%%%%%%

\section{Background material} \label{sec:backg}

In this section, we provide a background discussion of Hopf algebra actions and of tensor categories.

%%%%%%%%%%%%%%%%%%%%%%%%%%%%%%%%%%%%%%%

\subsection{Hopf algebra actions} \label{ssec:H-actions}

Let $H$ be a finite dimensional Hopf algebra over $k$ with comultiplication $\Delta$, counit $\epsilon$, and antipode $S$. Moreover, let
$H^*$ denote the dual Hopf algebra to $H$.
 A {\it left $H$-module} $M$ has left $H$-action denoted by $\cdot : H
\otimes M \rightarrow M$. Similarly, a {\it right $H$-comodule} $M$ has
right $H$-coaction denoted by $\rho: M \rightarrow M \otimes H$. Since
$H$ is finite dimensional,  $M$ is a left $H$-module if and
only if $M$ is a right $H^*$-comodule.

We recall basic facts about Hopf algebra actions; refer to
\cite{Montgomery} for further details. Denote the Hopf algebra $H^*$ by $K$. Here, $H$ will {\it act on} algebras (from the left), whereas $K$ will {\it coact on} algebras (from the right).

\begin{definition} \label{def:Hopfact} Given a Hopf algebra $H$ and an algebra $A$, we say
that {\it $H$ acts on $A$} (from the left) if $A$ is a left $H$-module and
$$h \cdot (ab) = \sum (h_1 \cdot a)(h_2 \cdot b)
\text{\hspace{.1in} and \hspace{.1in}} h \cdot 1_A = \epsilon(h)
1_A$$ for all $h \in H$, and for all $a,b \in A$, where $\Delta(h) =
\sum h_1 \otimes h_2$ (Sweedler's notation).

Dually, a Hopf algebra $K$ {\it coacts on} $A$ (from the right) if $A$ is a right $K$-comodule and $\rho(ab) = \rho(a) \rho(b)$ for all $a,b \in A$. 
\end{definition}

We also want to restrict ourselves to $H$-actions (and $K$-coactions) that do not factor
through `smaller' Hopf algebras.

\begin{definition}\cite{BanicaBichon} \label{def:infaith} Given a left $H$-module $M$, we say that $M$ is {\it
inner faithful} over $H$, if $IM\neq 0$ for every nonzero Hopf ideal
$I$ of $H$. Given a Hopf action of $H$ on an algebra $A$, we say that this
{\it action is  inner faithful} if the left $H$-module $A$ is inner
faithful.

Dually, a right $K$-comodule $M$ is called {\it inner faithful} if
for any proper Hopf subalgebra $K'\subsetneq K$, we have that $\rho(M)$ is not contained in
$M\otimes K'$. In this case, we say that the $K$-coaction on $M$ is
{\it inner faithful}. Similarly, we can define an {\it inner faithful $K$-coaction} on an algebra $A$.
\end{definition}

Note that an $H$-action on an algebra $A$ is inner faithful if and only if the
$H^*=K$-coaction on $A$ is inner faithful.

%%%%%%%%%%%%%%%%%%%%%%%%%%%%%%%%%%%%%%%

\subsection{Tensor categories} \label{ssec:tencat}

We discuss the notion of a {\it tensor category} and, in particular, of a {\it fusion category}, below. This will mainly be used in the proof of Theorem~\ref{thm:sscoideal} in Section~\ref{sec:coideal}. Refer to \cite{BakalovKirillov,ENO:fusion} for further details.
Recall that $k$ is an algebraically closed field of characteristic zero.

\begin{definition} \label{def:tensorcat}
Let $\mc{C}$ be a $k$-linear, abelian, rigid, monoidal category with unit object {\bf 1} that is artinian (so, objects have finite length) and Hom spaces are finite dimensional. 
% A finite category is locally finite, has finitely many simple objects, and has enough projectives. So, a finite category is equivalent to the category of finite dimensional modules over a finite dimensional algebra.
We call $\mc{C}$ a {\it tensor category} over $k$ if the bifunctor $\otimes: \mc{C} \times \mc{C} \ra \mc{C}$ is bilinear on morphisms and if $\End_{\mc{C}}({\bf 1}) \cong k$.

Recall that an abelian category is {\it semisimple} if every object is a direct sum of  simple objects, and is {\it finite} if it has enough projectives and has finitely many simple objects up to isomorphism. Now, a {\it fusion category} is a finite semisimple tensor category.
\end{definition}

For example, given $H$ a finite dimensional Hopf algebra over $k$, the category $\Rep(H)$ of (left) $H$-modules is a finite tensor category. Moreover, $\Rep(H)$ is a fusion category precisely when $H$ is semisimple.

We have that tensor categories (resp., fusion categories) categorify the notion of  rings (resp., semisimple rings). Similarly,  the notion of a {\it module category}  categorifies the concept of a module over a ring. By a module category, we mean a {\it right module category}.

\begin{definition} \label{def:modulecat} Let $\mc{C}$ be a tensor category with {\it associativity constraint} 
$$a_{X,Y,Z}: (X \otimes Y) \otimes Z \overset{\sim}{\lra} X \otimes (Y \otimes Z)$$ for all $X,Y,Z \in \mc{C}$ and unit object {\bf 1}.
A {\it right module category} over  $\mc{C}$ is a category $\mc{M}$ equipped with an {\it action bifunctor} $\otimes: \mc{M} \times \mc{C} \ra \mc{M}$ and a natural isomorphism
$$m_{M,X,Y}: M \otimes (X \otimes Y) \overset{\sim}{\lra} (M \otimes X) \otimes Y$$ for all $X,Y \in \mc{C}$ and $M \in \mc{M}$ so that
\begin{itemize}
\item the functor $M \mapsto M \otimes {\bf 1}: \mc{M} \ra \mc{M}$ is an autoequivalence and
\item the following {\it pentagon diagram}
\[
\xymatrix{
& M \otimes ((X \otimes Y) \otimes Z) 
\ar[dl]_{id_M \otimes a_{X,Y,Z}} 
\ar[dr]^{m_{M,X \otimes Y,Z}}&\\
M \otimes (X \otimes (Y \otimes Z))
\ar[d]_{m_{M,X,Y\otimes Z}}&&
(M \otimes (X \otimes Y)) \otimes Z
\ar[d]^{m_{M,X,Y\otimes id_Z}}\\
(M \otimes X) \otimes (Y \otimes Z) 
\ar[rr]^{m_{M\otimes X,Y,Z}}&&
((M \otimes X) \otimes Y) \otimes Z 
}
\]
is commutative for all $X,Y,Z \in \mc{C}$ and $M \in \mc{M}$.
\footnote{ This definition is equivalent to the usual definition of a module category involving the unit morphism and the triangle diagram.}
\end{itemize}

\end{definition}

Moreover, we also need to consider the following definitions pertaining to tensor categories $\mc{C}$ and module categories $\mc{M}$ over $\mc{C}$.

\begin{definition} Let $\mc{A}, \mc{B}$ be artinian abelian categories and  let $\mc{C}, \mc{D}$ be tensor categories. Also, let $\mc{M}, \mc{N}$ be module categories over $\mc{C}$ with associativity constraints $m,n$, respectively.
\begin{enumerate}
\item  An exact functor $G: {\mathcal A} \ra {\mathcal B}$ is {\it surjective} if any
object $Y \in {\mathcal B}$ is a subquotient of $G(X)$ for some $X \in {\mathcal A}$.
\item A {\it tensor functor} $F: \mc{C} \ra \mc{D}$ is an exact, faithful, $k$-linear, monoidal functor between tensor categories.
\item A $\mc{C}$-{\it module functor} from $\mc{M}$ to $\mc{N}$ consists of an additive functor $F:\mc{M} \ra \mc{N}$ and a natural isomorphism 
$$s_{M,X}: F(M \otimes X) \ra F(M) \otimes X, ~~\text{for all~} X \in \mc{C}, M \in \mc{M},$$
so that the following equations hold:
$$ (s_{M,X} \otimes id_Y) \circ s_{M\otimes X,Y} \circ F(m_{M,X,Y}) = n_{F(M),X,Y} \circ s_{M,X \otimes Y}$$
$$F(r_M) = r_{F(M)} \circ s_{M,{\bf 1}}$$
for all $X,Y \in \mc{C}$ and $M \in \mc{M}$. Here, $r_M: M \otimes {\bf 1} \overset{\sim}{\ra} M$ is the natural isomorphism.
\item Let Fun$_{\mc{C}}$($\mc{M}$,$\mc{N}$) denote the category of $\mc{C}$-module functors $\mc{M}\ra \mc{N}$ with natural transformations as morphisms.
\item A {\it module equivalence} $F: \mc{M} \ra \mc{N}$ of $\mc{C}$-module categories is a module functor ($F,s$) from $\mc{M}$ to $\mc{N}$ such that $F$ is an equivalence of categories.
\item A module category $\mc{M}$ over $\mc{C}$ is {\it indecomposable} if it is not equivalent to a nontrivial direct sum of nonzero module categories.
\end{enumerate}
\end{definition}
 
For example, consider $\Vec$, the category of finite dimensional vector spaces over $k$.
If $F: {\mathcal C} \ra \Vec$ is a tensor functor, then $\Vec$ becomes a (right) module category over ${\mathcal C}$ via
$V\otimes X:=V\otimes F(X)$, for $V\in \Vec$ and  $X\in {\mathcal C}$.

%%%%%%%%%%%%%%%%%%%%%%%%%%%%%%%%%%%%%%%
%%%%%%%%%%%%%%%%%%%%%%%%%%%%%%%%%%%%%%%
%%%%%%%%%%%%%%%%%%%%%%%%%%%%%%%%%%%%%%%

\section{Coideal subalgebras of finite dimensional Hopf algebras} \label{sec:coideal}

In this section, we establish several results pertaining to the number and structure of coideal subalgebras of finite dimensional Hopf algebras. In particular, we show that semisimple Hopf algebras have only finitely many coideal subalgebras; see Theorem~\ref{thm:sscoideal} below. Such a result fails for nonsemisimple Hopf algebras; see Example~\ref{ex:Sweedler}.

\begin{definition} \label{def:coidsubalg} 
A {\it right (respectively, left) coideal subalgebra} $B$ of a Hopf algebra $H$ is a subalgebra of $H$ with $\Delta(B) \subseteq B \otimes H$ (or $\Delta(B) \subseteq H \otimes B$).
\end{definition}

We will now show how coideal subalgebras could arise from coactions of Hopf algebras on commutative algebras. Consider the notation below.

\medskip
\noindent {\it Notation}. [$A$, $\chi$, $\rho_{\chi}$, $A_{\chi}$]
Let $A$ be a finitely generated commutative domain over $k$ and let $K$ be a finite dimensional Hopf algebra that coacts on $A$ via $\rho: A \ra A \otimes K$. Moreover, let $\chi:A \ra k$ be a character of $A$. Then, consider the following morphism
$$\rho_{\chi}=  (\chi \otimes id) \circ \rho: A \ra K.$$
Here, we identify $k \otimes K$ with $K$. We also  take $A_{\chi}$ to be the image of $\rho_{\chi}$ in $K$.
\medskip

\begin{lemma} \label{lem:Achi} 
Retain the notation above. The image $\rho_{\chi}(A) = A_{\chi}$ is a right coideal subalgebra of $K$.
\end{lemma}

\begin{proof}
Since $\rho_{\chi}$ is an algebra morphism, $A_{\chi}$ is a subalgebra of $K$. So, it suffices to show that $\Delta \circ \rho_{\chi} = (\rho_{\chi} \otimes id) \circ \rho$ as morphisms. We have that

\begin{equation} \label{eq:rho_chi}
\begin{array}{rl}
(\rho_{\chi} \otimes id) \circ \rho 
&= (\chi \otimes id \otimes id) \circ (\rho \otimes id) \circ \rho\\
&= (\chi \otimes id \otimes id) \circ (id \otimes \Delta) \circ \rho
= \Delta  \circ (\chi \otimes id) \circ \rho
= \Delta \circ \rho_{\chi},
\end{array}
\end{equation}

as desired.
\end{proof}

Now, we consider the coaction of the Hopf subalgebra generated by the span of the coideal subalgebras $A_{\chi}$ of $K$ on $A$.

\medskip

\noindent {\it Notation}. [$L_A$] Given a $K=H^*$-coaction on $A$, let $L_A$ be the $k$-linear span of all coideal subalgebras $A_{\chi}$ of $K$. 

\medskip

\begin{lemma} \label{lem:L_A} Retain the notation above.
\begin{enumerate}
\item The coaction $\rho$ of $K$ on $A$ restricts to $\rho: A \ra A \otimes L_A$.
\item The linear span $L_A$ is a subcoalgebra of $K$. Thus,
the subalgebra $\langle L_A \rangle$ generated by $L_A$ is a Hopf subalgebra of $K$.
\end{enumerate}
\end{lemma}

\begin{proof}
(a) It suffices to show that for any $\psi \in L_A^{\perp} \subseteq K^*$, we get that $$(id \otimes \psi) \circ \rho(a) = 0$$ for all $a \in A$. Fix an element $a \in A$. Let $X = \Specm(A)$ be the set of characters (or, equivalently, maximal ideals) in $A$. Then, $X$ is an affine algebraic variety and $f:= (id \otimes \psi) \rho(a) \in A$ can be viewed as a regular function on $X$. Recall that $A$ is a finitely generated domain. Now by the Nullstellensatz, to check that $f=0$, it suffices to check that $f(\chi) = \psi(\rho_{\chi}(a)) =0$ for all $\chi \in X$. So we need $\rho_{\chi}(a) \in L_A^{\perp \perp} = L_A$. Since $\rho_{\chi}(a) \in A_{\chi}$ and $A_{\chi} \subseteq L_A$, we are done.

\medskip

\noindent (b) We need to show for any $b \in A_{\chi}$, we have that $\Delta(b) \in L_A \otimes L_A$. Say $$b = \rho_{\chi}(a) = (\chi \otimes id) \circ \rho(a).$$ By (\ref{eq:rho_chi}), we have that
$\Delta(b) = (\rho_{\chi} \otimes id) \circ \rho(a)$. Now by part (a), $\Delta(b) \subseteq A_{\chi} \otimes L_A$ as desired.
This implies that $L_A$ is a subcoalgebra of $K$. Hence, $\langle L_A \rangle$ is a subbialgebra of $K$, which is actually a Hopf subalgebra of $K$ by \cite[Proposition~7.6.1]{Radford:book}.
\end{proof}

In the following example, we show how one computes the coideal subalgebras $A_{\chi}$ corresponding to a finite dimensional Hopf algebra $K$ and a $K$-comodule algebra $A$; refer to Lemma~\ref{lem:Achi}. As a consequence, we also illustrate that a nonsemisimple finite dimensional Hopf algebra can have infinitely many coideal subalgebras.

\begin{example} \label{ex:Sweedler} 
Let $H$ be the Sweedler Hopf algebra, which is 4-dimensional and non-semisimple.  It is generated by $g$ and $x$ where
\[
\begin{array}{lllll}
g^2=1, & x^2=0, & xg=-gx, & \Delta(g)=g \otimes g, & \Delta(x) = g \otimes x + x \otimes 1,\\
\epsilon(g)=1, & \epsilon(x)=0, & S(g)=g^{-1}, & S(x)=-gx.
\end{array}
\]
Let $A$ be the commutative polynomial ring $k[u]$ with left $H$-action given by
$$g \cdot u = -u \hspace{.2in}\text{and} \hspace{.2in} x \cdot u = 1.$$
It is easy to check that this action is inner-faithful. 
Moreover, the action $\cdot$ yields the right coaction of $K=H^*$ on $A$ given by 
$\rho: A \ra A \otimes K$, where 
$$\rho(u) = u \otimes (1^* - g^*) + 1 \otimes (x^* +(gx)^*).$$
Here, $\{1^*, g^*, x^*, (gx)^*\}$ is the dual basis of $K$. Note that $H$ is self-dual, so $H \cong K$ as Hopf algebras.

To define a set of right coideal subalgebras of $K$ (and thus of $H$),  consider the character $\chi: A \ra k$ defined by $\chi(u) = \alpha \in k$. Moreover, consider the morphism $\rho_{\chi}: A \ra K$ defined by 
$$\rho_{\chi}(u) = \alpha(1^* - g^*) + (x^* + (gx)^*).$$ 
Take $A_{\chi}$ to be the image of $\rho_{\chi}$, which is spanned by $1_{K}$ and $h_{\chi} := \alpha(1^* -g^*) + (x^* + (gx)^*)$.
Here, $1^* + g^*$ is the unit $1_K$ of $K$ since $1^* + g^* = \epsilon$ is the counit of $H$.

We show explicitly that the image $A_{\chi}$ is a right coideal subalgebra of $K$  for all $\alpha \in k$. First, $A_{\chi}$ is clearly a subalgebra of $K$. Secondly, recall that $\Delta(f) = \sum f(e_i e_j) e_i^* \otimes e_j^*$ for all $e_i, e_j \in H$ and $f \in K=H^*$. Let $\bar{g}:=1^*-g^*$ and $\bar{x}:=x^*+ (gx)^*$, so $h_{\chi} = \alpha \bar{g} + \bar{x}$. Moreover, $\Delta(\bar{g}) = \bar{g} \otimes \bar{g}$ and $\Delta(\bar{x}) = 1_K \otimes \bar{x} + \bar{x} \otimes \bar{g}$. 
So, we get that for all $\alpha \in k$:
$$\Delta(h_{\chi}) = \alpha(\bar{g} \otimes \bar{g}) + (1_K \otimes \bar{x} + \bar{x} \otimes \bar{g}) = h_{\chi} \otimes \bar{g} + 1_K \otimes \bar{x} \in A_{\chi} \otimes K.$$
Therefore, the Sweedler Hopf algebra has infinitely many right coideal subalgebras.
\end{example}

\medskip

On the other hand, we establish the following theorem pertaining to the number of coideal subalgebras of a semisimple Hopf algebra.

\begin{theorem} \label{thm:sscoideal} 
Let $k$ be a field of characteristic zero. Then, a semisimple Hopf algebra $K$ over $k$ has finitely many coideal subalgebras.
\end{theorem}

\begin{remark} \label{rmk:finite}
This theorem is one of the many finiteness (or ``rigidity") results for both semisimple Hopf algebras and fusion categories. These include the Ocneanu rigidity theorem (there are finitely many fusion categories with a given fusion rule) and Stefan's theorem (there are finitely many semisimple Hopf algebras of a given dimension over a field of characteristic 0). Such theorems are discussed in \cite{ENO:fusion}.
\end{remark}

To prove Theorem~\ref{thm:sscoideal}, we need the following preliminary result.

\begin{lemma} \label{lem:Bss}  \cite[Lemma~4.0.2]{Burciu:kernels} \cite[Theorem~5.2]{Skryabin:freeness}
Any left or right coideal subalgebra $B$ of a semisimple Hopf algebra $K$ is semisimple. \qed
%Moreover, $H$ is a free module over $B$.
\end{lemma}

%\begin{proof}
%This follows from \cite[Lemma~4.0.2]{Burciu:kernels}. 
%and \cite[Theorem~6.1(ii)]{Skryabin:freeness} verifies the second statement.
%\end{proof}

Now, we use the machinery of fusion categories to prove Theorem~\ref{thm:sscoideal}.

\medskip
\noindent {\it Notation.} [$B$, $\mc{C}$, $\mc{M}$, $F_{\mc{C}}$, $F_{\mc{M}}$, $G$, $\sigma$]
Let $K$ be a Hopf algebra and let $B$ be a right coideal subalgebra of $K$.  Consider the fusion category $\mc{C}:=\Rep(K)$ of left $K$-modules and  the category $\mc{M}:=\Rep(B)$ of left $B$-modules. We see that $\mc{M}$ is a right $\mc{C}$-module category as follows.  Given $X \in \mc{C}$ and $M \in  \mc{M}$, we get an action of $B$ on $M \otimes X$ by
$$b \cdot (m \otimes x) = \sum (b_{1} \cdot m) \otimes (b_{2} \cdot x),$$
where $\Delta(b) = \sum  b_{1} \otimes b_{2}$ for $b, b_{1} \in B$ and $b_{2} \in K$.
We also have that if $K$ is semisimple, then $\mc{M}$  is semisimple by Lemma~\ref{lem:Bss}.

Moreover, we have a functor $G: \mc{C} \ra \mc{M}$ defined by restriction from $K$ to $B$, which is a surjective module functor. We also have forgetful functors, $F_{\mc{C}}: \mc{C} \ra \Vec$ (a tensor functor) and $F_{\mc{M}}: \mc{M} \ra \Vec$ (a module functor), where $\Vec$ is the category of finite dimensional vector spaces over $k$. We also get an isomorphism of module functors over $\mc{C}$ given by $\sigma:  F_{\mc{M}} \circ G \overset{\sim}{\ra} F_{\mc{C}}$.
\medskip

Next, we establish a bijection between the set of right coideal subalgebras of $K$ and the set of quadruples $(\mc{M}, F_{\mc{M}}, G, \sigma)$ up to equivalence. By equivalence, we mean the equivalence relation generated by the following conditions.
\begin{enumerate}
\item[(1)] If $L: \mc{M} \ra \mc{M}'$ is an equivalence of right $\mc{C}$-module categories with quasiinverse $L^{-1}$, then $(\mc{M}, F_{\mc{M}}, G, \sigma)$ is equivalent to $(\mc{M}', F_{\mc{M}} \circ L^{-1}, L \circ G, \sigma')$ where $\sigma'$ is the corresponding isomorphism.
\item[(2)] If $G': \mc{C} \ra \mc{M}$ and $F_{\mc{M}}': \mc{M} \ra \Vec$ are other $\mc{C}$-module functors with isomorphisms $a: G \ra G'$ and $b: F_{\mc{M}} \ra F_{\mc{M}}'$, then $(\mc{M}, F_{\mc{M}}, G, \sigma)$ is equivalent to $(\mc{M}, F_{\mc{M}}', G', \sigma')$ where $\sigma'$ is the corresponding isomorphism.
\item[(3)] The quadruple $(\mc{M}, F_{\mc{M}}, G, \sigma)$ is equivalent to $(\mc{M}, F_{\mc{M}}, G, \lambda \sigma)$ with $\lambda \in k^{\times}$.
\end{enumerate}

\begin{lemma} \label{lem:bij} Retain the notation above.
Then, the following statements hold.
\begin{enumerate}
\item There is a bijection between the set of right coideal subalgebras $B$ of $K$ and the set of quadruples ($\mc{M}$, $F_{\mc{M}}$, $G$, $\sigma$) up to equivalence.
\item The module category $\mc{M}$ is indecomposable.
\end{enumerate}
\end{lemma}

\begin{proof}
(a) Our job is to show that given a semisimple right $\mc{C}$-module category $\mc{M}$ and a diagram
\[
\xymatrix{
\mc{C} \ar[r]^{G} \ar@/_3pc/[rr]_{F_{\mc{C}}} 
&\mc{M} \ar[r]^{F_{\mc{M}}} \ar@{-->}[d]^{\sigma} &\Vec\\
&{} &
}
\]
\begin{center} Diagram 1 \end{center}

\noindent where $G$ and $F_{\mc{M}}$ are module functors, with $G$ surjective, and $\sigma$  an isomorphism of module functors, we can construct a unique coideal subalgebra $B$ of $K$.

We have a homomorphism $\phi$ from the algebra, $\End(F_{\mc{M}})$, of functorial endomorphisms of the functor $F_{\mc{M}}$ to the Hopf algebra $K$ defined as follows:
$$\phi: \End(F_{\mc{M}}) \overset{G}{\lra} \End(F_{\mc{M}}\circ G) \overset{Ad(\sigma)}{\lra} \End(F_{\mc{C}}) = K.$$
 The last equality holds by the reconstruction theorem for Hopf algebras \cite{JoyalStreet}. Since $G$ is surjective, the map $\phi$ is injective and the image $B$ of $\phi$ can be viewed as a subalgebra of $K$.

We see that $B$ is a coideal subalgebra of $K$ as follows. Fix elements $m \in M \in \mc{M}$ and $x \in X \in \mc{C}$. (We abuse notation by writing $M$ for $F_{\mc{M}}(M)$ and $X$ for $F_{\mc{C}}(X)$ as actions technically occur in $F_{\mc{M}}(M)$ and $F_{\mc{C}}(X)$, respectively.) Consider the coaction $\rho: B \ra B \otimes K$ defined by 
$\rho(b) = \sum b_{1} \otimes b_{2}$, where 
$$b \cdot (m \otimes x)  = \sum (b_{1} \cdot m) \otimes (b_{2} \cdot x).$$ 
This makes sense since $M \otimes X \in \Rep(B)$. Here, $\mc{M}$ is naturally identified with $\Rep(\End(F_{\mc{M}}))$ by the reconstruction theorem for associative algebras.
Moreover, the map $\rho$ is identified with the endomorphism algebra of the functor $(M,X) \mapsto F_{\mc M}(M) \otimes F_{\mc C}(X)$ on the product category $\mc{M} \times \mc{C}.$
Now, it suffices to show that $\rho = \Delta|_B$; we verify this in the following sublemma.

\begin{sublemma} \label{sublem}
The coproduct $\Delta$ of $K$ restricted to $B$ is given by the $K$-coaction $\rho$ on $B$.
\end{sublemma}

\noindent {\it Proof of Sublemma~\ref{sublem}}. Consider the following standard isomorphisms:
\[
\begin{array}{c}
J_{X,Y}: F_{\mc{C}}(X) \otimes F_{\mc{C}}(Y) \overset{\sim}{\lra} F_{\mc{C}}(X \otimes Y),\\
s_{X,Y}: G(X \otimes Y) \overset{\sim}{\lra} G(X) \otimes Y,\\
r_{M,Y}: F_{\mc{M}}(M) \otimes F_{\mc{C}}(Y) \overset{\sim}{\lra} F_{\mc{M}}(M \otimes Y)
\end{array}
\]
for all $X, Y \in \mc{C}$ and $M \in \mc{M}$. Here, $J$ is the tensor structure on $F_{\mc{C}}$, and $s$ and $r$ are the structures of a module functor on $G$ and $F_{\mc{M}}$, respectively. Also, we have the isomorphism $\sigma: F_{\mc{C}} \ra F_{\mc{M}} \circ G$.
Now, the following diagram commutes for all $X, Y \in \mc{C}$ and $M \in \mc{M}$:
\[
\xymatrix{
F_{\mc{M}}(G(X \otimes Y)) 
 \ar[d]^{\sim}_{\sigma} 
&& F_{\mc{M}}(G(X) \otimes Y)   \ar[ll]^{\sim}_{F_{\mc{M}}(s_{X,Y})}
&& F_{\mc{M}}(G(X)) \otimes F_{\mc{C}}(Y) 
\ar[ll]^{\sim}_{r_{G(X),Y}} \ar[d]^{\sim}_{\sigma} \\
F_{\mc{C}}(X \otimes Y)
&&&& F_{\mc{C}}(X) \otimes F_{\mc{C}}(Y) \ar[llll]^{\sim}_{J_{X,Y}}
}
\]
Let $b \in \End(F_{\mc{M}})$. We leave it as an exercise to show that
\begin{enumerate}
\item[(i)] the image of $b$ under conjugation by $J_{X,Y}^{-1} \circ \sigma$ is given by $\Delta(\phi(b))$; and
\item[(ii)] the image of $b$ under conjugation by
$\sigma \circ r_{G(X),Y}^{-1} \circ F_{\mc{M}}(s_{X,Y}^{-1})$ is given by $(\phi \otimes id)\rho(b)$.
\end{enumerate}
Thus, $\Delta(\phi(b)) = (\phi \otimes id) \rho(b)$, and we are done.
\smallskip

Returning to the proof, we claim that $\sigma$ is unique up to scaling. It suffices to show that any automorphism of the module functor $F_{\mc{C}}$ is a scalar. An automorphism of $F_{\mc{C}}$, as an additive functor, is just an element $h \in K$. The condition that it preserves module structure is $\Delta(h)=h \otimes 1$, which implies that $h$ is a scalar.

It is clear that $\phi$ does not change under rescaling $\sigma$. So, since $\sigma$ is unique up to scaling, $\phi$ (and hence, the image $B$ of $\phi$) is independent of $\sigma$.
Finally, it is easy to check that the assignments 
$$B \mapsto (\Rep(B), F_{\Rep(B)}, G, \sigma) 
\hspace{0.2in}\text{and} \hspace{0.2in} 
(\mc{M}, F_{\mc{M}}, G, \sigma) \mapsto \phi(\End (F_{\mc{M}}))$$ are mutually inverse.

\medskip

\noindent (b) We claim that $\mc{M}$ is an indecomposable module category.
Let $J = G({\bf 1})$, which is simple as $F_{\mc{M}}(J) = F_{\mc{C}}({\bf 1}) = k$. Let $M \in \mc{M}$ be a simple object. Since $G$ is surjective, $M$ is contained in $G(X)$ for some $X \in \mc{C}$.  Moreover, 
$$G(X) = G({\bf 1} \otimes X) = G({\bf 1})  \otimes X =  J \otimes X.$$ 
So, $M$ is contained in $J \otimes X$.  In other words, all simple objects $M$ of $\mc{M}$ are accessible from $J$. Hence, $\mc{M}$ is indecomposable.
\end{proof}

To prove Theorem~\ref{thm:sscoideal}, we will also need the following proposition.

\begin{proposition} \label{pro:finite} Let $\mc{C}$ be a fusion category.
\begin{enumerate}
\item There are finitely many semisimple indecomposable module categories $\mc{M}$ over  $\mc{C}$.
\item If $\mc{M}$ and $\mc{N}$ are semisimple finite module categories over $\mc{C}$, then there are finitely many module functors $F: \mc{M} \ra \mc{N}$ up to isomorphism, inducing a given map of Grothendieck groups Gr($\mc{M}$)$\ra$Gr($\mc{N}$).
\end{enumerate}
\end{proposition}

\begin{proof}
Part (a) follows from \cite[Corollary~2.35]{ENO:fusion}.  Since Fun$_{\mc{C}}$($\mc{M},\mc{N}$) is a semisimple abelian category with finitely many simple objects, part (b) holds 
\cite[Theorem~2.16]{ENO:fusion}.
\end{proof}

\medskip

\noindent {\it Proof of Theorem \ref{thm:sscoideal}}.
 By extension of the ground field, we may assume that $k$ is algebraically closed.
Consider the category of left $K$-modules $\mc{C} = \Rep(K)$, which is a fusion category as $K$ is semisimple.  By Proposition~\ref{pro:finite}(a), there are finitely many semisimple indecomposable module categories $\mc{M}$ over $\mc{C}$. So by Lemma~\ref{lem:bij}(a,b), we need to show that for all such $\mc{M}$, there are finitely many choices of $G$ and $F_{\mc{M}}$ as in Diagram~1. 

Let $\{X_i\}$ be the simple objects of $\mc{C}$ and let $\{M_j\}$ be the simple objects of $\mc{M}$. Then, the map of Grothendieck groups from $\mc{C}$ to $\mc{M}$ is determined by $G(X_i) = \bigoplus_j a_{ij} M_j$.  Since $G$ is surjective, for every $j$, there is an $i$ such that $a_{ij} >0$. Similarly, we have that $F_{\mc{M}}(M_j) = k^{d_j}$ for some $d_j > 0$. Hence, the equations $\dim_k X_i = \sum_j a_{ij} d_j$ have finitely many suitable solutions $(a_{ij}, d_j)$. So, there are finitely many suitable maps between Grothendieck groups from Gr($\mc{C}$) to Gr($\mc{M}$) and from Gr($\mc{M}$) to Gr($\Vec$)=$\mathbb{Z}$. Since $\mc{C}$, $\mc{M}$, and $\Vec$ are all semisimple finite module categories over $\mc{C}$, Proposition~\ref{pro:finite}(b) then implies that there are finitely many choices for both $G$ and $F_{\mc{M}}$. 
\qed
\medskip

Additionally, we have a result that is easily obtained from Theorem \ref{thm:sscoideal}.

\medskip

\noindent {\it Notation}. [CS$_d(K)$] Given a finite dimensional Hopf algebra $K$, let CS$_d(K)$ denote the variety of coideal subalgebras of $K$ of dimension $d$. 
\medskip

\begin{corollary} \label{cor:sscodieal'} 
Let $K$ be a finite dimensional Hopf algebra and let Gr$_d(K)$ be the Grassmannian of $d$-dimensional subspaces of $K$. 
Then,
CS$_d(K)$ is a closed subvariety of Gr$_d(K)$.
If $K$ semisimple, then CS$_d(K)$  consists of finitely many points. 
\end{corollary}

\begin{proof}
For a subspace of $K$ to be a coideal and a subalgebra of $K$ are closed conditions, so the first statement is clear. The second statement follows directly from Theorem~\ref{thm:sscoideal}.
\end{proof}

%%%%%%%%%%%%%%%%%%%%%%%%%%%%%%%%%%%%%%%
%%%%%%%%%%%%%%%%%%%%%%%%%%%%%%%%%%%%%%%
%%%%%%%%%%%%%%%%%%%%%%%%%%%%%%%%%%%%%%%

\section{Proof of Theorem \ref{thm:main}} \label{sec:proofmain}

This section is dedicated to the proof of our main theorem; see Theorem~\ref{thm:mainK} below. We also discuss the various ways this result fails if its hypotheses are amended; see Remarks~\ref{rmk:Sweedler} and~\ref{rmk:nilpotents}.

\begin{theorem} \label{thm:mainK}
If a semisimple Hopf algebra $K$ over $k$ coacts inner faithfully on a commutative domain $A$ over $k$, then $K$ is itself commutative. Thus, the coaction of $K$ reduces to the action of a finite group on $A$.
\end{theorem}

\begin{proof}
First, let us reduce to the case where $A$ is finitely generated. Any  $K$-comodule algebra $A$ is a union of finitely generated subalgebras invariant under this coaction. We see this as follows. Let $a \in A$, so $\rho(a) = \sum a_i \otimes h_i$ for $a_i \in A$ and $h_i \in K$. Let $C(a)$ be the $k$-linear span of $\{a_i\}$. Then, $C(a)$ contains $a$ as $a= \sum \epsilon(h_i) a_i$. Moreover, $C(a)$ is a finite dimensional $K$-subcomodule of $A$. So, the algebra $A(a) \subseteq A$ generated by $C(a)$ is a finitely generated right $K$-comodule subalgebra of $A$ containing $a$. Thus, $A$ is the union of all $A(a)$, which are finitely generated $K$-comodule algebras.

So, assume that $A$ is finitely generated and let $X = \Specm(A)$, whose closed points consist of characters $\chi:A \ra k$ of $A$. Then, $X$ is an irreducible affine algebraic variety over $k$. We have a map 
$$\gamma: X \lra \bigsqcup_d \text{CS}_d(K)$$
defined by $\gamma(\chi)=A_{\chi}$ (see the notation from Section~\ref{sec:coideal}).

Let $d_0 = \max_{\chi \in X} \dim_k A_{\chi}$. Consider the set $$X_0 = \{\chi \in X ~|~ \dim_k A_{\chi} = d_0\},$$ which is non-empty. The dimension map $g:X \ra [0, d_0]$ given by $g(\chi) = \dim_k A_{\chi}$ is lower semi-continuous, 
%rank functions x \mapsto dim(image(A(x)) are lower semi-continuous
so $X_0$ is also an open subset of $X$. Thus, $X_0$ is irreducible.

Now, we show that the map $\gamma|_{X_0}$ is regular. Take $a_1, \dots, a_m$ to be generators of $A$. Let $n$ be the dimension of $K$ and let $A(n)$ be the $k$-span of monomials $a_1^{i_1} \cdots a_m^{i_m}$ where $0 \leq i_1, \dots, i_m < n$. Since any element $x$ of $K$ satisfies a monic polynomial equation of degree $n$ (namely, the characteristic polynomial of the linear operator given by left multiplication of $x$), we have that $\rho_{\chi} (A(n)) = \rho_{\chi} (A)  = A_{\chi}$. Moreover,  the map $f:X_0 \ra \Hom_k(A(n), K)$ given by $f(\chi) = \rho_{\chi}|_{A(n)}$ is regular with the rank of $f(\chi)$ constant (independent of $\chi$). Since $\rho_{\chi}(A) = \rho_{\chi}(A(n))$, we have that im($f(\chi)$) = $\gamma(\chi)$ for any $\chi \in X_0$. This implies that $\gamma|_{X_0}$ is regular.

Since CS$_{d_0}(K)$ is finite by Theorem~\ref{thm:sscoideal}, $X_0$ is irreducible, and $\gamma|_{X_0}$ is regular, we have that $\gamma|_{X_0}$ is constant. In other words for all $\chi \in X_0$, we have that $\gamma(\chi) =B$ for some coideal subalgebra $B$ of $K$ whose dimension is maximal among the dimensions of the $A_{\chi}$. We see that $A_{\chi} \subseteq B$ for all $\chi \in X$ as follows. Let $\beta \in B^{\perp}$ and $a \in A$. Then, $\beta(\rho_{\chi}(a))$ is a regular function with respect to $\chi$, and it is zero for $\chi \in X_0$. Hence, this function is identically zero since $X_0$ is dense in $X$. Hence, $\rho_{\chi}(A) \subseteq B$ as claimed.

On the other hand, consider $ L_A$, the $k$-linear span of the coideal subalgebras $A_{\chi}$ of $K$; refer to Lemma~\ref{lem:L_A}. Since $A_{\chi} \subseteq B$ for all $\chi \in X$, we have that $L_A  \subseteq B$. Also, $B = A_{\chi}$ for some $\chi \in X_0$. Hence, $B = L_A$. So, $B$ equals the subalgebra  $\langle L_A \rangle$ generated by $L_A$, which is also a Hopf subalgebra of $K$ (Lemma~\ref{lem:L_A}(b)).

Again by Lemma~\ref{lem:L_A}(b), the coaction of $K$ on $A$ restricts to the coaction of $\langle L_A \rangle$ on $A$. By inner faithfulness, no proper Hopf subalgebra of $K$ can coact on $A$, so $\langle L_A \rangle =K$. Since $B =L_A= A_{\chi}$ for some $\chi \in X_0$ and $A_{\chi}$ is commutative, we have that $K = \langle L_A \rangle$ is commutative as desired.

The second statement of the theorem is clear.
\end{proof}

The following remarks illustrate how Theorem~\ref{thm:mainK} fails if one of its hypotheses is altered.
 
\begin{remark} \label{rmk:Sweedler}
The proof of Theorem~\ref{thm:mainK} fails if $K$ has infinitely many coideal subalgebras. In this case, $K$ must be nonsemisimple by Theorem~\ref{thm:sscoideal}. For example, consider the coaction of the (dual of the) Sweedler Hopf algebra $K$ on $A = k[u]$ from Example~\ref{ex:Sweedler}. Note that $K$ is a 4-dimensional vector space spanned by $1_K, \bar{g}, \bar{x}, \overline{gx}$, where $\bar{g}:= 1^* - g^*$ and $\bar{x}:=x^* + (gx)^*$.

Recall that the coideal subalgebras $A_{\chi}$ of $K$ are of the form 
$$\{ A_{\chi} \} = \{\langle  1_K , h_{\chi} := \alpha \bar{g} + \bar{x} \rangle ~|~ \alpha \in k\},$$
which all have $k$-vector space dimension 2.
The $k$-linear span, $L_A$, is spanned by $1_K, \bar{g}, \bar{x}$, which is a 3-dimensional $k$-vector space. However, $K = \langle L_A \rangle$ is 4-dimensional. Observe that each $A_{\chi}$ is commutative, but these coideal subalgebras do not commute with each other. Hence, $K$ is noncommutative.
\end{remark}

\begin{remark} \label{rmk:nilpotents}
Theorem~\ref{thm:mainK} also fails if $A$ is not a domain. First, let $K = kS_3$, the group algebra of the symmetric group $S_3$, with $s_1 := (12)$ and $s_2:=(23)$. Let $A = k[u_1, u_2]/(u_1^2, u_1u_2, u_2^2)$. Here, $A$ contains nonzero nilpotents, yet $A$/Rad($A$) is a domain. Although $K$ is noncommutative, we can define an inner faithful coaction of $K$ on $A$ by $\rho(u_i) = u_i \otimes s_i$ for $i = 1,2.$

Secondly, let $K$ be as above, and let $A'=k[u_1,u_2]/(u_1 u_2)$. The algebra $A'$ has zero divisors, yet no nonzero nilpotents. Consider the $K$-coaction on $A'$ given by $\rho(u_i) = u_i \otimes s_i$ for $i=1,2$. Again, the $K$-coaction on $A'$ is inner faithful, but $K$ is noncommutative.
\end{remark}
%%%%%%%%%%%%%%%%%%%%%%%%%%%%%%%%%%%%%%%
%%%%%%%%%%%%%%%%%%%%%%%%%%%%%%%%%%%%%%%
%%%%%%%%%%%%%%%%%%%%%%%%%%%%%%%%%%%%%%%

\section{Consequences and further directions} \label{sec:directions}

Here, we discuss consequences of Theorem~\ref{thm:main}, which include versions of the main theorems for $k$ a field of characteristic $p>0$ and a study of Hopf actions on Weyl algebras. We also present further directions of this work in the last subsection.

%%%%%%%%%%%%%%%%%%%%%%%%%%%%%%%%%%%%%%

\subsection{Results for $k$ a field of positive characteristic} \label{ssec:charp}

We can generalize the main results of this work for when the field $k$ is algebraically closed of characteristic $p>0$. 

\begin{theorem} \label{thm:mainKcharp} 
Theorem~\ref{thm:mainK} holds if the algebraically closed field $k$ has characteristic $p>0$ and $K$ is a semisimple and cosemisimple Hopf algebra.
\end{theorem}

\begin{proof}
In Section~9 of \cite{ENO:fusion}, it is explained that Proposition~\ref{pro:finite} extends to positive characteristic if the fusion category $\mc{C}$ is nondegenerate (see e.g. \cite[Theorem~9.3]{ENO:fusion}). This is the case if $\mc{C} = \Rep(K)$, where $K$ is a semisimple and cosemisimple Hopf algebra over $k$. The rest of the proof is the same as in characteristic zero.
\end{proof}

By similar arguments, we also have the following theorem. 

\begin{theorem} \label{thm:sscoidealcharp}
Theorem~\ref{thm:sscoideal} holds  if the algebraically closed field $k$ has characteristic $p>0$ and $K$ is a semisimple and cosemisimple Hopf algebra.
\end{theorem}

Furthermore, we make the following conjecture.

\begin{conjecture} \label{conj:charp}
Theorems \ref{thm:sscoideal} and \ref{thm:mainK} 
hold when we work over an algebraically closed field $k$ of characteristic $p >0$ and  $K$ is semisimple and not necessarily cosemisimple.
\end{conjecture}

%%%%%%%%%%%%%%%%%%%%%%%%%%%%%%%%%%%%%%

\subsection{Hopf actions on differential operator algebras} \label{ssec:Weyl}

It was asked in \cite[Question~0.2]{CWWZ:filtered} whether a noncocommutative finite-dimensional Hopf algebra can act inner faithfully on the $n$-th Weyl algebra. This question was answered negatively for $n=1$. Now, the general result is obtained from Theorem~\ref{thm:main} as follows.
First, we require some preliminary results. We assume that the filtrations below are indexed by nonnegative integers.

\begin{proposition} \label{pro:filtered}
Let $S$ be a filtered $k$-algebra with filtration $F$ so that the associated graded ring, gr$_F S$, is a commutative domain over $k$.  If a semisimple Hopf algebra $H$ acts on $S$ inner faithfully and preserves the filtration, then $H$ is a group algebra.
\end{proposition}

\begin{proof}
Since $H$ is semisimple, we have that $S$ and gr$_F S$ are isomorphic as $H$-modules. So, the induced $H$-action on the commutative domain gr$_F S$ is inner faithful. By Theorem~\ref{thm:main}, $H$ is a group algebra.
\end{proof}

Now we answer \cite[Question~0.2]{CWWZ:filtered}.

\begin{corollary} \label{cor:Weyl}
Let $H$ be a finite dimensional Hopf algebra acting inner faithfully on the $n$-th Weyl algebra $A_n(k)$ with the standard filtration. If the $H$-action preserves the filtration of $A_n(k)$, then $H$ is a finite group algebra.
\end{corollary}

\begin{proof}
The standard filtration $F$ of $A_n(k)$ is given by $\{F_n = (k1 + U)^n\},$ where $U$ is the $2n$-dimensional vector space spanned by the generators $u_1, \dots, u_n,$ $v_1, \dots, v_n$ of $A_n(k)$. Here, 
$[u_i, u_j] = [v_i, v_j] = 0$ and $[v_i, u_j] = \delta_{ij}$. By \cite[Theorem~0.3]{CWWZ:filtered}, we have that if $H$ satisfies the hypotheses above, then $H$ is semisimple. The result follows from Proposition~\ref{pro:filtered} as  gr$_F A_n(k)$ is isomorphic to the commutative polynomial ring $k[u_1, \dots, u_n, v_1, \dots, v_n]$.
\end{proof}

We have a more general corollary to Proposition~\ref{pro:filtered}. Here, we use the notions of the so-called {\it homological determinant} of a Hopf action on a (graded) algebra $A$ and  of a (graded) $r$-{\it Nakayama}  algebra; see \cite{CWZ:Nakayama} for details.

\begin{corollary} \label{cor:filtered}
Let $H$ be a finite dimensional Hopf algebra that acts inner faithfully and preserves the filtration $F$ of a filtered algebra $S$. Assume the following conditions:
\begin{enumerate}
\item gr$_F S$ is a commutative domain, 
\item  the Rees ring, Rees$_F S$, is connected graded, $r$-Nakayama and $N$-Koszul, 
\item the induced $H$-action on Rees$_F S$ has trivial homological determinant.
\end{enumerate}
Then, $H$ is a group algebra.
\end{corollary}

\begin{proof}
The inner faithful $H$-action on $S$ induces an inner faithful $H$-action on Rees$_F S$. So,  $H$ must be semisimple by \cite[Theorem~0.6]{CWZ:Nakayama}. Since the inner faithful $H$-action on $S$ now induces an inner faithful $H$-action on gr$_F S$, we are done by Proposition~\ref{pro:filtered}.
\end{proof}

It would be interesting to know if the Weyl algebras are the only $k$-algebras that satisfy the hypotheses of Corollary~\ref{cor:filtered}. 
%\ff{What about $S= U(\mf{sl}_2$) = $k\langle e,f,h \rangle/( [e,f]= h, [h,e] = 2e, [h,f] = -2f)$ with the standard filtration, for example?}
On the other hand, Corollary~\ref{cor:Weyl} prompts the following question.

\begin{question} \label{ques:diffop}
Let $X$ be any smooth irreducible affine variety and consider the algebra of differential operators $D(X)$ on $X$. If a finite dimensional Hopf algebra acts inner faithfully and preserves the order filtration $F_{\text{ord}}$ of $D(X)$, must then $H$ be a group algebra?
\end{question}

Note that gr$_{F_{\text{ord}}} D(X)$ is isomorphic to the algebra of regular functions $O(T^* X)$, where $T^* X$ is the cotangent bundle on $X$. Hence, gr$_{F_{\text{ord}}} D(X)$ is a commutative domain. Moreover, Question~\ref{ques:diffop} is open even if $X =k^n$, to say, for $D(X) = A_n(k)$ with the order filtration.

%%%%%%%%%%%%%%%%%%%%%%%%%%%%%%%%%%%%%%

\subsection{Additional questions} \label{ssec:questions}
We pose the following questions for future work. 
To begin, note that the main theorem (Theorem~\ref{thm:main}) and Example~\ref{ex:Sweedler} naturally prompt the question below.

\begin{question} \label{ques:nonss} Which finite dimensional nonsemisimple Hopf algebras act inner faithfully on the commutative domains?
\end{question}

On the other hand, as an extension of Theorem~\ref{thm:main}, we consider Hopf actions on PI algebras.

\begin{question} \label{ques:PIalgebra}
If a cosemisimple Hopf algebra $H$ over $k$ acts inner faithfully on a PI domain of PI degree $d$, must then PIdeg($H^*$) $\leq d^2$?
\end{question}

If the answer is affirmative, then we have that the bound $d^2$ is sharp due to the following example.

\begin{example}
Let $\zeta$ be a primitive $d$-th root of unity and let $L$ denote the group $\Z_d \oplus \Z_d$. There is a nondegenerate 2-cocycle $\sigma$ on $L$, given by
$\sigma((x,y), (x', y')) = \zeta^{xy'},$
where $x,y,x',y' \in \Z_d$. 
Moreover, $\sigma$ defines a Drinfeld twist $J$ on $L$, given by $J = \sum_{x,y,x',y'} \zeta^{xy'} (x,y) \otimes (x',y')$. 

Take $G$ to be a finite group containing $L$ with an element $g$, such that $gLg^{-1} \cap L = \{1\}$. 
For instance, one could take $G = GL_n(\Z_d)$ for $n \geq 3$, where the embedding $\iota: L \ra G$ is given by $\iota(a,b) = Id + a E_{12} + b E_{13}$ for $a,b \in \Z_d$.
Now, by \cite[Theorem~3.2]{EtGel:cotriangular}, the PI degree of $((kG)^J)^*$ is equal to $|L| = d^2$.

Assuming that we have a faithful, linear action of $G$ on a commutative polynomial ring $A$ in $n$ variables, we have an inner faithful, linear action of $(kG)^J$ 
on the twisted algebra $A_J$ \cite{GKM:twist}. At least one of the skew parameters of the quantum polynomial ring $A_J$ is a primitive $d$-th root of unity, so $A_J$ has PI degree at least $d$. On the other hand, the rank of $A_J$ over (central invariants) $A_J^L$ is $|L|$, where $A_J^L\cong A^L$ as algebras. Hence, the PI degree of $A_J$ is at most $|L|^{1/2} = d$. 
\end{example}

%We have a positive answer for several examples, including
%inner faithful actions of the self-dual nontrivial 8-dimensional Hopf algebra $H_8$ on PI algebras $R = k\langle u,v \rangle/(vu-quv)$, where $q = -1, i, -i$ (\cite[Example~7.4]{KKZ:Gorenstein}, for instance). Here, $R$ has PI degree 2, 4, and 4, respectively, whereas PIdeg($H_8$)= 2.
%For other examples, consider the seven classes of semisimple Hopf ($H$-)actions on Artin-Schelter regular algebras $A$ of global dimension 2, where the inner faithful $H$-action on $A$ has {\it trivial homological determinant} \cite{CKWZ}. In the two cases where $H$ is not a group algebra, the algebra $A$ has PI degree 2. One of the Hopf algebras that acts on such $A$ is  $H= (kD_{2n})^*$, the dual of the dihedral group algebra for $n \geq 3$. Here, PIdeg$(kD_{2n})$ = 2. The other Hopf algebra that acts on such $A$ is $H = \mc{D}^*$, the dual of a Hopf deformation of a binary polyhedral group. By \cite[Remark~5.22]{BichonNatale},  $\mc{D}$ also has PI degree 2.
%
%
%

\section*{Acknowledgments} The authors thank the referee for providing suggestions that improved the exposition of this article.
The authors  would also  like to thank Ellen Kirkman and James Kuzmanovich for suggesting Question~\ref{ques:KK} to the second author.
This work was supported by the National Science Foundation: NSF-grants DMS-1000173 and DMS-1102548.

\bibliography{Hopf_biblio}

\def\cprime{$'$}\def\cprime{$'$}\def\cprime{$'$}\def\cprime{$'$}\def\cprime{$'$}\def\cprime{$'$}\def\cprime{$'$}
\begin{thebibliography}{CWWZ14}

\bibitem[All09]{Allman}
Justin Allman.
\newblock Actions of finite dimensional non-commutative {H}opf algebras on
  rings ({M}aster's thesis), 2009.

\bibitem[BB10]{BanicaBichon}
Teodor Banica and Julien Bichon.
\newblock Hopf images and inner faithful representations.
\newblock {\em Glasg. Math. J.}, 52(3):677--703, 2010.

\bibitem[BK01]{BakalovKirillov}
Bojko Bakalov and Alexander Kirillov, Jr.
\newblock {\em Lectures on tensor categories and modular functors}, volume~21
  of {\em University Lecture Series}.
\newblock American Mathematical Society, Providence, RI, 2001.

\bibitem[Bur12]{Burciu:kernels}
Sebastian Burciu.
\newblock Kernels of representations and coideal subalgebras of {H}opf
  algebras.
\newblock {\em Glasg. Math. J.}, 54(1):107--119, 2012.

\bibitem[Coh94]{Cohen:survey}
M.~Cohen.
\newblock Quantum commutativity and central invariants.
\newblock In {\em Advances in {H}opf algebras ({C}hicago, {IL}, 1992)}, volume
  158 of {\em Lecture Notes in Pure and Appl. Math.}, pages 25--38. Dekker, New
  York, 1994.

\bibitem[CWWZ14]{CWWZ:filtered}
K.~Chan, C.~Walton, Y.~Wang, and J.~J. Zhang.
\newblock Hopf actions on filtered regular algebras.
\newblock {\em J. Algebra}, 397(1):68--90, 2014.

\bibitem[CWZ]{CWZ:Nakayama}
Kenneth Chan, Chelsea Walton, and J.~J. Zhang.
\newblock Hopf actions and {N}akayama automorphisms (preprint).
\newblock {\em {\tt http://arxiv.org/abs/1210.6432}, submitted}.

\bibitem[EG99]{EtGel:cotriangular}
Pavel Etingof and Shlomo Gelaki.
\newblock The representation theory of cotriangular semisimple {H}opf algebras.
\newblock {\em Internat. Math. Res. Notices}, (7):387--394, 1999.

\bibitem[ENO05]{ENO:fusion}
Pavel Etingof, Dmitri Nikshych, and Viktor Ostrik.
\newblock On fusion categories.
\newblock {\em Ann. of Math. (2)}, 162(2):581--642, 2005.

\bibitem[GKM12]{GKM:twist}
Pierre Guillot, Christian Kassel, and Akira Masuoka.
\newblock Twisting algebras using non-commutative torsors: explicit
  computations.
\newblock {\em Math. Z.}, 271(3-4):789--818, 2012.

\bibitem[JS91]{JoyalStreet}
Andr{\'e} Joyal and Ross Street.
\newblock An introduction to {T}annaka duality and quantum groups.
\newblock In {\em Category theory ({C}omo, 1990)}, volume 1488 of {\em Lecture
  Notes in Math.}, pages 413--492. Springer, Berlin, 1991.

\bibitem[Mon93]{Montgomery}
Susan Montgomery.
\newblock {\em Hopf algebras and their actions on rings}, volume~82 of {\em
  CBMS Regional Conference Series in Mathematics}.
\newblock Published for the Conference Board of the Mathematical Sciences,
  Washington, DC, 1993.

\bibitem[Rad12]{Radford:book}
David~E. Radford.
\newblock {\em Hopf algebras}, volume~49 of {\em Series on Knots and
  Everything}.
\newblock World Scientific Publishing Co. Pte. Ltd., Hackensack, NJ, 2012.

\bibitem[Skr07]{Skryabin:freeness}
Serge Skryabin.
\newblock Projectivity and freeness over comodule algebras.
\newblock {\em Trans. Amer. Math. Soc.}, 359(6):2597--2623 (electronic), 2007.

\end{thebibliography}

\end{document}